\documentclass[12pt]{amsart}

\usepackage{amssymb,latexsym,amsmath,amsthm,fullpage}
\usepackage{graphics,graphicx,xypic}

\newcommand{\D}{\mathrm{d}}
\newcommand{\der}{\mathrm{der}}

\newcommand{\Aff}{\mathrm{Aff}}
\newcommand{\Stab}{\mathrm{Stab}}

\newcommand{\clos}[1]{\overline{#1}}
\newcommand{\til}{\widetilde}
\newcommand{\R}{\mathbb{R}}
\newcommand{\C}{\mathbb{C}}
\newcommand{\Z}{\mathbb{Z}}

\newcommand{\eps}{\varepsilon}
\newcommand{\Lie}[1]{\mathrm{#1}}

\newtheorem{theorem}{Theorem}
\newtheorem{lemma}{Lemma}[section]
\newtheorem{proposition}{Proposition}[section]

\theoremstyle{definition}

\newtheorem{example}{Example}

\theoremstyle{remark}

\newtheorem*{remark*}{Remark}
\newtheorem*{notation*}{Notation}

\numberwithin{equation}{section}

\begin{document}

\title{Finiteness conditions on translation surfaces}

\author{Joshua P.~Bowman}
\address{IMS, Department of Mathematics \\
Stony Brook University \\
Stony Brook, NY 11794}
\email{{\tt joshua.bowman@gmail.com}}


\date{}


\maketitle

\section*{Introduction}

Throughout this note, let $X$ denote a {\em translation surface}, i.e., 
a (connected) topological surface with a translation atlas. Then $X$ is 
automatically endowed with a conformal structure and a flat metric, and 
so it is both a Riemann surface and a Riemannian manifold \cite{jHhM79}. 
An orientation-preserving homeomorphism $\phi : X \to X$ is called 
{\em affine} if it is affine in local charts. We use $\Aff^+(X)$ to denote 
the group of affine maps of $X$. Any element $\phi$ of $\Aff^+(X)$ has 
a well-defined global {\em derivative} $\der\,\phi \in \Lie{GL}_2^+(\R)$. 
The image $\Gamma(X)$ of the homomorphism 
$\der : \Aff^+(X) \to \Lie{GL}_2^+(\R)$ is called the {\em Veech group} 
of $X$ \cite{wV89,yV96,cEfG97,eGcJ00}.

The existence of affine self-maps of a translation surface has applications 
in the study of mapping class groups, Teichm{\"u}ller theory, algebraic 
geometry, and dynamical systems (for a small sampling of such applications, 
see, e.g., \cite{wpT88,pHtS00,ctM03,mM06,cLaR06,ldM11}). They measure a 
kind of ``symmetry'' more general than that of isometries, which nonetheless 
has consequences for such systems as geodesic flow on the surface and 
geodesics in Teichm{\"u}ller space. Veech first observed the importance of 
the group of derivatives of affine maps \cite{wV89}.

Let $\clos{X}$ denote the metric completion of $X$. The classical study of 
translation surfaces assumes that $\clos{X}$ is itself a compact surface 
and $\clos{X} \setminus X$ is finite. If these conditions are satisfied, 
we will say that $X$ has {\em finite affine type}. Here we wish to consider 
four other ``finiteness'' conditions that may be placed on $X$:
\begin{enumerate}
\item\label{I:an} $X$ has {\em finite analytic type} as a Riemann surface, 
meaning that it is obtained from a compact Riemann surface by making 
finitely many punctures.
\item\label{I:ar} $X$ has {\em finite area} as a Riemannian manifold, 
meaning that the integral of the induced area form over all of $X$ is 
finite.
\item\label{I:bd} $X$ is {\em bounded} as a metric space, meaning that 
there exists a constant $M > 0$ such that $d_X(x,y) \le M$ for every pair 
of points $x$ and $y$ in $X$.
\item\label{I:tb} $X$ is {\em totally bounded} as a metric space, meaning 
that for any fixed $\eps > 0$, $X$ can be covered by finitely many balls 
of radius $\eps$ (equivalently, $\clos{X}$ is compact).
\end{enumerate}

We will prove two main results about these conditions, one negative and 
one positive.

\begin{theorem}\label{T:1}
Except for the trivial implication ``totally bounded $\implies$ bounded'', 
none of the conditions \eqref{I:an}--\eqref{I:tb} on $X$ implies any of 
the others. However, if $X$ has finite analytic type, then the other three 
conditions are equivalent and imply that $X$ has finite affine type.
\end{theorem}

\begin{theorem}\label{T:2}
Suppose the ideal boundary of $X$ is empty. If $X$ has at least one 
periodic trajectory and is totally bounded or has finite area, then its 
Veech group is a discrete subgroup of $\Lie{SL}_2(\R)$. However, there 
exist bounded surfaces and surfaces of finite analytic type with 
non-discrete Veech groups.
\end{theorem}

\begin{remark*}
It is likely that the condition of having a periodic trajectory follows 
from the assumptions of having empty ideal boundary and being totally 
bounded or of finite area, in which case it can be dropped in 
Theorem~\ref{T:2}.
\end{remark*}

Translation surfaces of infinite analytic type appear, for example, 
in \cite{cEfG97,rCfGnL06,pHsLsT11,pHpHbW12,fV12,jpb12}, and it is 
such examples that motivated the study presented here. We will prove 
Theorem~\ref{T:1} in \S\ref{S:inequiv} and Theorem~\ref{T:2} in 
\S\ref{S:discrete}.

\section{Inequivalence of finiteness conditions}\label{S:inequiv}

We begin with the trivial, and only, implication among the finiteness 
conditions \eqref{I:an}--\eqref{I:tb}.

\begin{proposition}\label{P:1.1}
``$X$ is totally bounded'' $\implies$ ``$X$ is bounded''.
\end{proposition}
\begin{proof}
This is a generality about metric spaces. Pick $\eps > 0$, and cover $X$ 
with $N$ balls of radius $\eps$. Then the distance between any two points 
is at most $2N\eps$.
\end{proof}

The rest of the first part of Theorem~\ref{T:1} is proved through a series 
of examples. One general construction will be quite useful and flexible, 
so we describe it first and establish some notation.

\begin{example}[An infinite ``stack of boxes'']
Let $H = \{h_n\}_{n=1}^\infty$ be a sequence of positive numbers, 
and let $W = \{w_n\}_{n=1}^\infty$ be a strictly decreasing sequence 
of positive numbers tending to zero. Then we construct a translation 
surface $X_{H,W}$ as follows (see Figure~\ref{F:XHW}):
\begin{itemize}
\item For each $n \ge 1$, let $R_n$ be a rectangle with horizontal side 
$w_n$ and vertical side $h_n$.
\item Place the sequence of rectangles in the plane $\R^2$, starting with 
$R_1$ having its lower left corner at the origin, and with $R_{n+1}$ 
immediately above $R_n$ so that its left edge is along the $y$-axis.
\item Identify the right and left sides of each $R_n$ with each other via 
horizontal translation, and identify the portion of the top of $R_n$ not 
covered by $R_{n+1}$ (of length $w_n - w_{n+1}$) with the portion of the 
bottom edge of $R_1$ directly below via vertical translation. (We omit the 
vertices.)
\end{itemize}
The genus of $X_{W,H}$ is infinite, as can be seen by considering the 
(pairwise non-homotopic) horizontal core curves of the $R_n$. The area of 
$X_{H,W}$ is
\[
\mathrm{Area}(X_{H,W}) 
= \sum_{n=1}^\infty \mathrm{Area}(R_n) 
= \sum_{n=1}^\infty h_n w_n.
\]
In particular, the area of $X_{H,W}$ is finite if $H$ and $W$ are sequences 
in $\ell^2$, but this is not necessary. Let $\clos{X}_{H,W}$ denote the 
metric completion of this surface.
\end{example}

\begin{figure}
\includegraphics{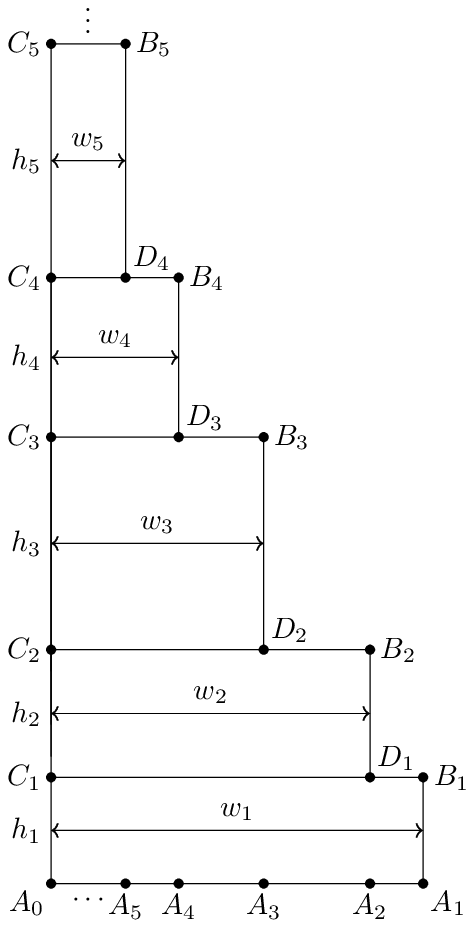}
\caption{A sample surface $X_{H,W}$. In addition to the identification of 
vertical edges, as indicated by the arrows, each segment $A_k A_{k+1}$ is 
identified with $B_k D_k$ via vertical translation.}
\label{F:XHW}
\end{figure}

\begin{lemma}
$\clos{X}_{H,W} \setminus X_{H,W}$ has only one point.
\end{lemma}
\begin{proof}
The translation structure has been defined by taking a quotient of the 
union of the rectangles $R_n$ except for their vertices. The vertices are 
all collapsed to a single point, as is evident in Figure~\ref{F:XHW}: 
using the notation of that figure, observe that $A_0 \sim A_1 \sim B_1 
\sim C_1 \sim D_1 \sim A_2 \sim B_2 \sim C_2 \sim D_2 \sim \cdots$.
\end{proof}

\begin{lemma}
$X_{H,W}$ is bounded if and only if $H$ is bounded.
\end{lemma}
\begin{proof}
Suppose $H$ is bounded by $M_H$ and $W$ by $M_W$. Then every point of 
every rectangle $R_n$ is within $M = \sqrt{{M_H}^2 + {M_W}^2}$ of a corner. 
Since the vertices are identified to a single point in $\clos{X}_{H,W}$, 
$2M$ is an upper bound for the distance between any two points of $X_{H,W}$.

Now suppose $H$ is not bounded. Then the centers of the rectangles $R_n$ 
become arbitrarily far from the vertices, and so $X_{H,W}$ is not bounded.
\end{proof}

\begin{lemma}
$X_{H,W}$ is totally bounded if and only if $H$ tends to zero.
\end{lemma}
\begin{proof}
Let $s$ denote the unique point in $\clos{X}_{H,W} \setminus X_{H,W}$.

Suppose $H$ tends to zero. Then, because $W$ also tends to zero, 
for every $\eps > 0$ there exists $N$ such that $h_n < \sqrt{\eps}$ 
and $w_n < \sqrt{\eps}$ for all $n \ge N$. This implies that the 
$\eps$-neighborhood $B_\eps$ of $s$ covers all $R_n$ for $n \ge N$. The 
complement of $B_\eps$ in $X_{H,W}$ is compact, being a finite union of 
compact pieces, and can therefore be covered by finitely many $\eps$-balls.

Suppose $H$ does not tend to zero. Then there exists $\eps_0 > 0$ such 
that $h_n > 2\eps_0$ for infinitely many $n \ge 1$. Fix such an $\eps_0$ 
and cover $\clos{X}_{H,W}$ by the following open sets:
\begin{itemize}
\item the $\eps_0$-neighborhood $B_{\eps_0}$ of $s$;
\item for each $R_n$ not fully covered by $B_{\eps_0}$ (of which there 
are infinitely many), the interior of this cylinder;
\item for each edge between $R_n$ and $R_{n+1}$, a neighborhood of radius 
$(1/3) \cdot \min\{h_n, h_{n+1}\}$.
\end{itemize}
Any finite subcover of this open cover would fail to cover infinitely many 
interior points of the $R_n$s, and so $\clos{X}_{H,W}$ cannot be compact.
\end{proof}

With these observations about $X_{H,W}$ in mind, we proceed to our 
counterexamples.

\begin{example}[\emph{finite area $\not\Rightarrow$ finite analytic type}] 
\label{Ex:aran}
Take $X_{H,W}$ with $h_n = w_n = 1/n$.
\end{example}

\begin{example}[\emph{finite area $\not\Rightarrow$ bounded}] 
\label{Ex:arid}
Take $X_{H,W}$ with $h_n = n$ and $w_n = 1/n^3$.
\end{example}

\begin{example}[\emph{finite area $\not\Rightarrow$ totally bounded}]
\label{Ex:artb}
Take $X_{H,W}$ with $h_n = 1$ and $w_n = 1/n^2$.
\end{example}

\begin{example}[\emph{bounded $\not\Rightarrow$ finite analytic type}]
Take Example \ref{Ex:aran} or \ref{Ex:artb}.
\end{example}

\begin{example}[\emph{bounded $\not\Rightarrow$ finite area}]
\label{Ex:bdar}
Take $X_{H,W}$ with $h_n = 1$ and $w_n = 1/n$.
\end{example}

\begin{example}[\emph{bounded $\not\Rightarrow$ totally bounded}]
Take Example \ref{Ex:artb} or \ref{Ex:bdar}.
\end{example}

\begin{example}[\emph{totally bounded $\not\Rightarrow$ finite analytic type}]
Take Example \ref{Ex:aran}.
\end{example}

\begin{example}[\emph{totally bounded $\not\Rightarrow$ finite area}]
\label{Ex:tbar}
Take $X_{H,W}$ with $h_n = w_n = 1/\sqrt{n}$.
\end{example}

\begin{example}[\emph{finite analytic type $\not\Rightarrow$ finite area or 
bounded}]
\label{Ex:anarbd}
The Riemann surface $\C^*$ has finite analytic type, since it is obtained 
from the Riemann sphere by removing two points. However, the translation 
structure given by the differential $\D{z}/z$ makes $\C^*$ isometric to an 
infinite cylinder, in which case it does not have finite area, and it is 
not bounded.
\end{example}

Example~\ref{Ex:anarbd} shows that the essential way a surface of finite 
analytic type can fail to have finite affine type is that one could take 
a {\em meromorphic} differential on a compact Riemann surface and remove 
the zeroes and poles to obtain a translation surface of finite analytic 
type. However, if we pair the ``finite analytic type'' condition with any 
of the others, then the rest follow. This fact is likely to be well-known, 
but we prove it here for completeness and to show how the analytic structure 
of the translation surface plays a role.

\begin{proposition}\label{P:1.2}
If $X$ has finite analytic type and finite area, then it is totally bounded.
\end{proposition}
\begin{proof}
The translation structure is given by an abelian differential on $X$. 
Let $\til{X}$ denote the compact surface from which $X$ is obtained as 
a Riemann surface. Because $X$ has finite area, the differential can be 
extended to $\til{X}$; each point of $\til{X} \setminus X$ is either a 
regular point or a zero of the differential. Thus $\clos{X}$ is canonically 
homeomorphic to $\til{X}$, so $X$ is totally bounded.
\end{proof}

\begin{proposition}\label{P:1.3}
If $X$ has finite analytic type and it is bounded, then it has finite area.
\end{proposition}
\begin{proof}
The translation structure is given by an abelian differential on $X$ that 
is meromorphic on the compact Riemann surface $\til{X}$ from which it is 
obtained by punctures. Because $X$ is bounded, none of the punctures is at 
an infinite distance from any other point of $X$. Therefore the differential 
has no poles on $\til{X}$, and so it has finite area.
\end{proof}

\begin{proof}[Proof of second part of Theorem~\ref{T:1}]
Immediate from Propositions \ref{P:1.1}, \ref{P:1.2}, and \ref{P:1.3}.
\end{proof}

To conclude this section, we observe that other collections of conditions 
do not imply any of the remaining ones, except as trivially follows from 
what has been established.

\begin{example}[\emph{finite area $+$ totally bounded $\not\Rightarrow$ 
finite analytic type}]
Take Example~\ref{Ex:aran}.
\end{example}

\begin{example}[\emph{finite area $+$ bounded $\not\Rightarrow$ totally bounded}]
Take Example~\ref{Ex:artb}.
\end{example}

\section{Discreteness of Veech groups}\label{S:discrete}

Recently, it has become apparent that translation surfaces of infinite 
analytic type allow for Veech groups of much greater complexity than occurs 
in the case of finite type. Specifically, it is well-known that the Veech 
group of a translation surface of finite affine type is always a Fuchsian 
(i.e., discrete) subgroup of $\Lie{SL}_2(\R)$ and is never co-compact. In 
contrast, it has been shown by direct construction that any countable 
subgroup of $\Lie{SL}_2(\R)$ (in fact, of $\Lie{GL}_2^+(\R)$) that avoids 
the set of matrices with operator norm less than $1$ can occur as the Veech 
group of a translation surface whose topological type is that of a ``Loch 
Ness Monster'', meaning it has infinite genus and one topological end 
\cite{pPgSfV11}. Other ``naturally occurring'' examples (e.g., the surface 
obtained by ``unfolding'' an irrational polygon \cite{fV12}) also 
demonstrate that one cannot in general expect the Veech group of a 
translation surface of infinite type to be discrete. In this section, 
we show that this phenomenon of non-discreteness relies essentially on 
the failure of a surface to be totally bounded or to have finite area; 
i.e., it is not enough that the surface simply have infinite analytic type.

The usual proof of discreteness in the case of finite affine type is 
carried out by showing that the Veech group acts on the set of holonomy 
vectors of saddle connections, which is a discrete subset of $\C$ (see, 
e.g., \cite{yV96}). For surfaces not of finite affine type, this last 
clause no longer holds: in many examples, the holonomy vectors of saddle 
connections do not have their lengths bounded away from zero. We find 
another subset of $\C$ on which the Veech group acts and which, under the 
conditions of Theorem~\ref{T:2}, is also discrete. Our proof holds also 
for surfaces of finite affine type, and bypasses considerations of whether 
the holonomy vectors of saddle connections form a discrete set or not.

Observe, first of all, that if $X$ has finite area, then any element of 
$\Aff^+(X)$ must preserve this area, and so the condition that 
$\Gamma(X) \subset \Lie{SL}_2(\R)$ follows automatically. Similarly, we 
have the following.

\begin{lemma}\label{L:2.1}
If $X$ is totally bounded, then $\Gamma(X) \subset \Lie{SL}_2(\R)$.
\end{lemma}
\begin{proof}
We use the compactness of $\clos{X}$ to establish a kind of Poincar{\'e} 
recurrence, which will permit us to define a first return map. Let 
$\phi \in \Aff^+(X)$. For any open subset $U$ of $X$ with piecewise smooth 
boundary, we observe that the images $\phi^{\circ n}(U)$ cannot all be 
disjoint: for otherwise, we could take them together with one more open 
subset, formed by the union of their complement and regular neighborhoods 
of their boundaries, and we would have an open cover of $\clos{X}$ with no 
finite subcover. Therefore, by a standard argument, 
$\phi^{\circ N}(U) \cap U \ne \varnothing$ for some $N \ge 1$. 
Proceeding inductively, we obtain a first return map $R_\phi$ into $U$, 
defined on an open subset of $U$ whose complement has measure zero. Choose 
$U$ so that it has finite area. The area of the image is 
\[
\mathrm{Area}(R_\phi) = \int_U (\det DR_\phi)\,d\mathrm{Area} 
\le \mathrm{Area}(U).
\]
If $\der\,\phi$ had determinant greater than $1$, then the Jacobian 
determinant in the above integral would be greater than $1$ on the entire 
domain, and the given inequality would not hold. We conclude that any 
element of $\Aff^+(X)$ must have a derivative in $\Lie{SL}_2(\R)$.
\end{proof}

Now we proceed to the main ideas in the proof of Theorem~\ref{T:2}.
Throughout this section, we take cylinders in $X$ to be {\em open} 
subsets of $X$; that is, they do not include their boundaries.

\begin{lemma}\label{L:2.2}
Let $C_1$ and $C_2$ be two maximal cylinders in a translation surface whose 
respective circumferences are $w_1$ and $w_2$ and whose respective heights 
are $h_1$ and $h_2$, and suppose that they intersect but do not coincide. 
Then the angle $\theta$ between the core curves of $C_1$ and $C_2$ satisfies 
\[
|\tan\theta| > \min \left\{ \frac{h_1}{w_1},\frac{h_2}{w_2} \right\}.
\]
\end{lemma}
\begin{proof}
Let $\gamma_1$ and $\gamma_2$ be the core curves of $C_1$ and $C_2$, 
respectively. If $\gamma_1$ and $\gamma_2$ meet at right angles, then we 
are done. So suppose they do not meet at right angles. We note that each 
time $\gamma_1$ crosses one boundary component of $C_2$, it must cross the 
other boundary component before returning to the first, and likewise for 
$\gamma_2$ crossing the boundary of $C_1$. Therefore the connected 
components of $C_1 \cap C_2$ are Euclidean parallelograms whose sides are 
arcs of the boundaries of $C_1$ and $C_2$. Let $P$ be one such 
parallelogram (see Figure~\ref{F:C1C2}). The angles of $P$ are $\theta$ 
and $\pi - \theta$, so it suffices to consider the smaller of these angles. 
Note that $h_1$ and $h_2$ are also the two heights of $P$. Let $l_1$ and 
$l_2$ be the distances from the vertex at $\theta$ to the orthogonal 
projections of the adjacent vertices onto the adjacent sides of $P$ in the 
directions of $\gamma_1$ and $\gamma_2$, respectively. At least one of the 
following inequalities holds: $l_1 < w_1$ or $l_2 < w_2$. But 
$\tan\theta = h_1/l_1 = h_2/l_2$, from which the desired result follows.
\end{proof}

\begin{figure}
\includegraphics{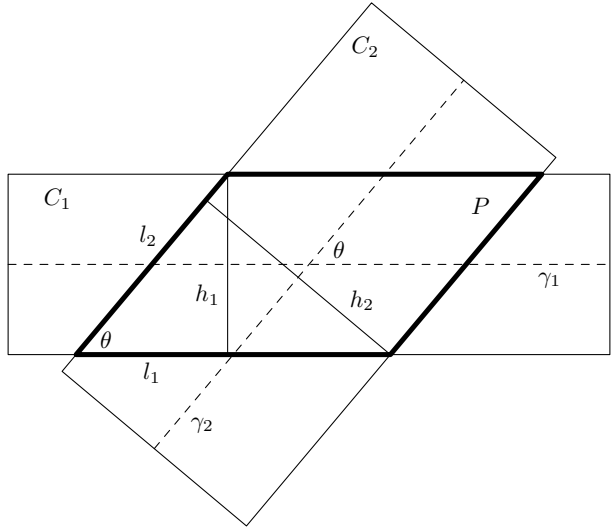}
\caption{Setup for the proof of Lemma~\ref{L:2.2}. The length of 
$\gamma_1$ is $w_1$, and the length of $\gamma_2$ is $w_2$.}
\label{F:C1C2}
\end{figure}

\begin{lemma}\label{L:2.3}
If $v_0, v \in \C$ satisfy $|v_0 - v| < \eps < |v_0|$, then the angle 
$\theta$ between $v_0$ and $v$ satisfies 
\[
|\tan\theta| < \frac{\eps}{\sqrt{|v_0|^2 - \eps^2}}.
\]
\end{lemma}
\begin{proof}
Under the given conditions, the largest angle a vector $v$ can make with 
$v_0$ is when $v$ is tangent to the circle with radius $\eps$ centered at 
$v_0$; the assumptions imply that this angle is strictly smaller than 
$\pi/2$ in absolute value. The result now follows by direct calculation 
of the tangent of the angle in this extreme case and monotonicity of the 
tangent function on $(-\pi/2,\pi/2)$.
\end{proof}

The values $h_1/w_1$ and $h_2/w_2$ in Lemma~\ref{L:2.2} are, of course, 
the {\em moduli} of the cylinders. The basic idea behind the next lemma 
is that if two cylinders have the same area and almost the same 
circumference, then their moduli are not very different; we can thus play 
the two inequalities of Lemmata \ref{L:2.2} and \ref{L:2.3} against each 
other.

\begin{notation*}
Given a translation surface $X$ and $A > 0$, we denote by 
$\mathcal{C}(A)$ the set of maximal cylinders on $X$ with area $A$, 
and by $\mathcal{V}(A) \subset \C$ the set of holonomy vectors of core 
curves of elements of $\mathcal{C}(A)$.
\end{notation*}

\begin{lemma}\label{L:2.4}
Let $X$ be a translation surface that either is totally bounded or has 
finite area, and let $A > 0$. Then $\mathcal{V}(A)$ is either empty or a 
discrete subset of $\C$.
\end{lemma}
\begin{proof}
Let $v_0 \in \mathcal{V}(A)$. If $0 < \eps < |v_0|$ and 
$v \in \mathcal{V}(A)$ is any vector such that $|v - v_0| < \eps$, 
then the modulus of any corresponding cylinder is bounded below by 
\[
f_1(\eps) = \frac{A}{(|v_0| + \eps)^2}.
\]
On the other hand, Lemma~\ref{L:2.3} implies that if $|v_0 - v| < \eps$ 
and $|v_0 - w| < \eps$, then the absolute value of the tangent of the 
angle between $v$ and $w$ is bounded above by 
\[
f_2(\eps) = \frac{2\eps \sqrt{|v_0|^2 - \eps^2}}{|v_0|^2 - 2\eps^2}.
\]
Note that, as $\eps \to 0$, $f_1(\eps)$ tends to $A/|v_0|^2$, while 
$f_2(\eps)$ tends to $0$. We can therefore choose $\eps_0 > 0$ small 
enough that $f_2(\eps_0) < f_1(\eps_0)$. Then Lemma~\ref{L:2.2} implies 
that, for any pair of distinct elements $v,w \in \mathcal{V}(A)$ such 
that $|v_0 - v| < \eps_0$ and $|v_0 - w| < \eps_0$, the corresponding 
cylinders in $\mathcal{C}(A)$ must be disjoint.

If $X$ has finite area, there can only be finitely many such cylinders, 
and so there can only be finitely many elements of $\mathcal{V}(A)$ within 
$\eps_0$ of $v_0$.

If $X$ is totally bounded, there again can be only finitely many such 
cylinders; otherwise we could find infinitely many disjoint balls of some 
fixed positive radius on $X$, which is impossible in a totally bounded 
space.

In either case, $v_0$ is an isolated point in $\mathcal{V}(A)$; since 
$v_0$ was arbitrary, $\mathcal{V}(A)$ is discrete.
\end{proof}

\begin{lemma}\label{L:2.5}
Let $X$ be a translation surface that either is totally bounded or has 
finite area, and let $A > 0$. Then $\Aff^+(X)$ preserves $\mathcal{C}(A)$ 
and $\Gamma$(X) preserves $\mathcal{V}(A)$.
\end{lemma}
\begin{proof}
The affine image of a cylinder is a cylinder, and the maximality of a 
cylinder is preserved because saddle connections are sent to saddle 
connections by elements of $\Aff^+(X)$. We have already observed that 
an affine self-homeomorphism of $X$ must preserve area, and so the first 
claim is proved. The second follows immediately.
\end{proof}

\begin{lemma}\label{L:2.6}
Let $X$ be a translation surface without ideal boundary, and let 
$C \subset X$ be a maximal cylinder of finite area. Suppose that $X$ 
is not a torus. Then the stabilizer $\Stab(C)$ of $C$ in $\Aff^+(X)$ 
is a cyclic group, hence discrete.
\end{lemma}
\begin{proof}
Because $C$ has finite area and $X$ is not a torus, $C$ has an ideal 
boundary. Because the ideal boundary of $X$ is empty, $X$ does not consist 
only of $C$. Therefore each boundary component of $C$ contains a saddle 
connection; call these $I_1$ and $I_2$. Any element of $\Stab(C)$ must 
also preserve the lengths of $I_1$ and $I_2$. Because the boundary of $C$ 
has finite length, we may assume, up to taking a finite index subgroup, 
that every element in $\Stab(C)$ fixes $I_1$ and $I_2$. But this implies 
that every element of $\Stab(C)$ fixes the entire boundary of $C$ and thus 
is a power of a full Dehn twist in $C$. Ergo $\Stab(C)$ is isomorphic to a 
subgroup of $\Z$, from which the result follows.
\end{proof}

\begin{proof}[Proof of first part of Theorem~\ref{T:2}]
The result is already known if $X$ is a torus, because then the Veech group 
is (conjugate to) $\Lie{SL}_2(\Z)$, so assume this is not the case.

Let $\gamma$ be a periodic trajectory on $X$. Then the image of $\gamma$ 
is contained in some maximal cylinder. This cylinder must have some finite 
area $A$: this is immediate if the area of $X$ is finite, and if $X$ is 
totally bounded, it follows because the points of the cylinder must be a 
bounded distance apart. Thus the set $\mathcal{V}(A)$ is non-empty. By 
Lemma~\ref{L:2.4}, it is therefore discrete. By Lemma~\ref{L:2.5}, 
$\Gamma(X)$ acts on $\mathcal{V}(A)$. So it suffices to show that the 
stabilizer inside $\Gamma(X)$ of a point $v \in \mathcal{V}(A)$ is discrete 
in $\Lie{SL}_2(\R)$. To see this, we observe that, up to taking a finite 
index subgroup, the stabilizer of $v$ in $\Gamma(X)$ may be identified with 
the stabilizer inside $\Aff^+(X)$ of some cylinder in $\mathcal{C}(A)$. 
Lemma~\ref{L:2.6} now implies the desired result.
\end{proof}

To prove the second part of Theorem~\ref{T:2}, we again turn to examples.

\begin{example}
Let $L$ be an ``irrational'' rhombus, meaning its angles are not rational 
multiples of $\pi$. The surface $X_L$ obtained by unfolding $L$ is such 
that $\clos{X}_L \setminus X_L$ consists of four points, arising from the 
vertices of $L$. $X_L$ is therefore bounded. Its Veech group, however, is 
an indiscrete subgroup of $\Lie{SO}(2)$, generated by rotations through the 
angles of $L$.
\end{example}

\begin{example}
The infinite cylinder of Example~\ref{Ex:anarbd} has finite analytic type. 
This surface has one homotopy class of periodic trajectories; these are 
the images of vertical lines in the plane under the universal covering map 
$\zeta \mapsto e^\zeta$ from $\C$ to $\C^*$, which is made into a 
translation covering by taking the differential $\D\zeta$ on the domain. 
The parabolic map $(x,y) \mapsto (x,y+tx)$ of $\R^2 \cong \C$ is affine 
with respect to $\D\zeta$ for any $t \in \R$, and it descends to $\C^*$ as 
an affine map with respect to $\D{z}/z$, acting as a Dehn twist on each 
annulus $\{ 2\pi k \le t \log|z| \le 2\pi(k+1) \}$, $k \in \Z$. The Veech 
group of $(\C^*,\D{z}/z)$ therefore contains a copy of $\R$, and so it is 
not discrete.
\end{example}

\begin{remark*}
It is still not known whether a surface of infinite genus that has finite 
area or is totally bounded can have a lattice Veech group---in particular, 
whether the Veech group of such a surface can be co-compact.
\end{remark*}

\subsection*{Acknowledgements}

The author wishes to thank the Hausdorff Research Institute for 
Mathematics in Bonn as well as the organizers of the trimester program 
``Geometry and Dynamics of Teichm{\"u}ller Spaces'', where much of 
this work was carried out. Thanks also to Pascal Hubert, Gabriela 
Schmith{\"u}sen, and Ferr{\'a}n Valdez for helpful conversations 
and feedback, and to the referee for useful suggestions.

\bibliographystyle{math}

\begin{thebibliography}{99}
\bibitem[Bo]{jpb12} 
Joshua P.~Bowman. The complete family of Arnoux--Yoccoz surfaces. 
To appear in {\em Geometriae Dedicata}.
\bibitem[CGL]{rCfGnL06} 
Reza Chamanara, Frederick P.~Gardiner, and Nikola Lakic. 
A hyperelliptic realization of the horseshoe and baker maps. 
{\em Erg.\ Theory Dynam.\ Systems} {\bf26} (2006), 1749--1768.
\bibitem[De]{ldM11} 
Laura DeMarco. The conformal geometry of billiards. 
{\em Bull. A.M.S.}\ {\bf48} (2011), 33--52.
\bibitem[EG]{cEfG97} 
Clifford J.~Earle and Frederick P.~Gardiner. 
{Teichm{\"u}ller} disks and {Veech}'s $\mathcal{F}$-structures. 
{\em Contemp.\ Math.}\ {\bf201} (1997), 165--189.
\bibitem[GJ]{eGcJ00} 
Eugene Gutkin and Chris Judge. 
Affine mappings of translation surfaces: Geometry and arithmetic.
{\em Duke Math.\ J.}\ {\bf103} (2000), 191--213.
\bibitem[HHW]{pHpHbW12}
W.~Patrick Hooper, Pascal Hubert, and Barak Weiss. 
Dynamics on the infinite staircase. 
To appear in {\em Discrete and Continuous Dynamical Systems -- Series A}.
\bibitem[HM]{jHhM79}
John H.~Hubbard and Howard Masur. 
Quadratic differentials and foliations. 
{\em Acta Math.}\ {\bf142} (1979), 221--274.
\bibitem[HLT]{pHsLsT11}
Pascal Hubert, Samuel Leli{\`e}vre, and Serge Troubetzkoy. 
The Ehrenfest wind-tree model: periodic directions, recurrence, diffusion. 
{\em J.\ reine angew.\ Math.}\ {\bf656} (2011) 223--244.
\bibitem[HS]{pHtS00}
Pascal Hubert and Thomas Schmidt. 
Veech groups and polygonal coverings. 
{\em J.\ Geom.\ and Phys.}\ {\bf35} (2000), 75--91.
\bibitem[LR]{cLaR06}
Christopher J.~Leininger and Alan W.~Reid. 
A combination theorem for Veech subgroups of the mapping class group. 
{\em Geom.\ Funct.\ Analysis} {\bf16} (2006), 403--436.
\bibitem[Mc]{ctM03}
Curtis T.~McMullen. 
Billiards and {Teichm{\"u}ller} curves on {Hilbert} modular surfaces. 
{\em J.\ Amer.\ Math.\ Soc.}\ {\bf16} (2003), 857--885.
\bibitem[{M\"o}]{mM06}
Martin M{\"o}ller. 
Variations of Hodge structures of a {Teichm{\"u}ller} curve. 
{\em J.\ Amer.\ Math.\ Soc.}\ {\bf19} (2006), 327--344.
\bibitem[PSV]{pPgSfV11}
Piotr Przytycki, Gabriela Schmith{\"u}sen, and Ferr{\'a}n Valdez. 
Veech groups of Loch Ness monsters. 
{\em Ann.\ Inst.\ Fourier} {\bf61} (2011), 673--687.
\bibitem[Th]{wpT88}
William P.~Thurston. 
On the geometry and dynamics of diffeomorphisms of surfaces. 
{\em Bull.\ A.M.S.}\ {\bf19} (1988), 417--431.
\bibitem[Va]{fV12}
Ferr{\'a}n Valdez.
Veech groups, irrational billiards and stable abelian differentials.
{\em Discrete and Continuous Dynamical Systems -- Series A} 
{\bf32} (2012), 1055--1063.
\bibitem[Ve]{wV89}
William A.~Veech.
{Teichm{\"u}ller} curves in moduli space, Eisenstein series, and an 
application to triangular billiards. 
{\em Inv.\ Math.}\ {\bf97} (1989), 553--583.
\bibitem[Vo]{yV96}
Yaroslav B.~Vorobets. 
Planar structures and billiards in rational polygons: 
the Veech alternative. 
{\em Russ.\ Math.\ Surv.}\ {\bf51} (1996), 779--817.
\end{thebibliography}

\end{document}